\documentclass[12pt]{amsart}

\usepackage{amssymb,amsmath, mathrsfs,amsfonts}

\usepackage{enumitem}

\usepackage{graphicx}

\makeatletter
\@namedef{subjclassname@2020}{%
	\textup{2020} Mathematics Subject Classification}
\makeatother

\usepackage[T1]{fontenc}

\newtheorem{theorem}{Theorem}[section]

\newtheorem{proposition}[theorem]{Proposition}

\theoremstyle{definition}

\numberwithin{equation}{section}

\begin{document}
	
	
	\baselineskip=17pt
	
	
	\title[An Exponential Diophantine equation $x^2+3^{\alpha} 113^{\beta}=y^{\mathfrak{n}}$]{An Exponential Diophantine equation $x^2+3^{\alpha} 113^{\beta}=y^{\mathfrak{n}}$}
	
	\author[S.Muthuvel]{S.Muthuvel}
	
	\address{Department of Mathematics, Faculty of Engineering and Technology, SRM Institute of Science and Technology, Vadapalani Campus, No.1 Jawaharlal Nehru Salai, Vadapalani, Chennai-600026, Tamil Nadu, India.}
	\email{muthushan15@gmail.com, ms3081@srmist.edu.in}
	
	\author[R.Venkatraman]{R.Venkatraman}
	\address{Department of Mathematics, Faculty of Engineering and Technology, SRM Institute of Science and Technology, Vadapalani Campus, No.1 Jawaharlal Nehru Salai, Vadapalani, Chennai-600026, Tamil Nadu, India.}
	\email{venkatrr1@srmist.edu.in}
	
	\date{}
	
	\begin{abstract}
		The objective of the paper is to determine the complete solutions for the Diophantine equation $x^2 + 3^{\alpha}113^{\beta} = y^{\mathfrak{n}}$ in positive integers $x$ and $y$ (where $x, y \geq 1$), non-negative exponents $\alpha$ and $\beta$, and an integer $\mathfrak{n}\geq 3$, subject to the condition $\text{gcd}(x, y) = 1$.
	\end{abstract}
	
	
	
	\maketitle
	
	\section{Introduction}
	\noindent Consider the Diophantine equation of the form
	\begin{eqnarray}
		x^2+\mathcal{C}=y^\mathfrak{n}, \quad x,y \geq 1, \quad \mathfrak{n} \geq 3, \label{main-1}
	\end{eqnarray}
	where $\mathcal{C}$ is a fixed positive integer. The initial discovery of positive integer solutions for the given equation dates back almost 17 decades ago \cite{23}. It has been proven that the equation invariably has only a finite number of solutions that are positive integers \cite{21}. Early investigations focused on Eq. \eqref{main-1} when $\mathcal{C}=c_{0}$ is a constant integer \cite{20,29,30}. In \cite{16}, Cohn presented a solution to equation \eqref{main-1} with the stipulation that $gcd(x,y)=1$, while taking into account the parameter $\mathcal{C}$ within the interval $1 \leq \mathcal{C} \leq 100$, excluding particular $\mathcal{C}$ values. The work in \cite{28} addressed additional $\mathcal{C}$ values within the same range, and the remaining values were addressed in \cite{10}. Over the years, researchers have investigated not only instances where $\mathcal{C}=\mathfrak{p}^k$ with a specific prime number $\mathfrak{p}$ \cite{1,2,3,4,5,15,27} but also for the case of a general prime number $\mathfrak{p}$ \cite{6,8,22,31,36}.
	
	Consider a set of primes $\mathcal{S}=\{\mathfrak{p}_{1},\mathfrak{p}_{2},\ldots,\mathfrak{p}_{k}\}$. Recent investigations focus on Eq. \eqref{main-1}, specifically when $\mathcal{C}$ is the product of prime powers $\mathfrak{p}^{k}$, with $\mathfrak{p} \in \mathcal{S}$ for any non-negative integer $k$ \cite{alan,  9,12,13,17,18,19,24,25,26,Muthu_exp_397,32,33,rayaguru,34,35}. Additionally, in \cite{37}, the authors explored positive integer solutions to Eq. \eqref{main-1} in a more general context, considering $\mathcal{C}=2^a \mathfrak{p}^b$ for any odd prime $\mathfrak{p}$.
	
	The primary focus of this paper is the Diophantine equation given by 
	\begin{eqnarray} 
		x^2+3^{\alpha} 113^{\beta}=y^{\mathfrak{n}} , \ \mathfrak{n} \geq 3,   \label{1} 
	\end{eqnarray}
	where $\alpha,\beta \geq 0 ~\text{and}~ x,y \geq 1 ~\text{with}~ \text{gcd}(x,y)=1$.\\
	Now, we proceed to establish the subsequent result.
	\begin{theorem}
		The equation \eqref{1} has the following solutions:
		\begin{align*}
			(x,y,\mathfrak{n},\alpha,\beta)=& (2, 7, 3, 1, 1), (1232, 115, 3, 3, 1), (23642486, 82375, 3, 9, 1),\\
			&(46, 13, 3, 4, 0), (10, 7, 3, 5, 0)
		\end{align*}
		excluding the case where $19 \nmid \mathfrak{n}$.
	\end{theorem}
	
	\section{Preliminaries}
	Consider algebraic integers $\eta$ and $\overline{\eta}$. A Lucas pair, represented as $(\eta, \overline{\eta})$, is defined by the conditions that $\eta + \overline{\eta}$ and $\eta\overline{\eta}$ are non-zero coprime rational integers, and $\displaystyle \frac{\eta}{\overline{\eta}}$ is not a root of unity. \\
	The sequences of Lucas numbers are defined in correspondence with any given Lucas pair $(\eta, \overline{\eta})$ as follows:
	\begin{eqnarray*}
		\mathcal{L}_{\mathfrak{n}}(\eta,\overline{\eta})=\frac{\eta^{\mathfrak{n}}-\overline{\eta}^{\mathfrak{n}}}{\eta-\overline{\eta}}, \quad \mathfrak{n}=0,1,2,\ldots
	\end{eqnarray*}
	
	The existence of primitive divisors for $\mathcal{L}_{\mathfrak{n}}(\eta,\overline{\eta})$ is a crucial concept in the realm of Lucas sequences.
	
	A prime number $\mathfrak{p}$ is defined as a primitive divisor of $\mathcal{L}_{\mathfrak{n}}(\eta,\overline{\eta})$ if $\mathfrak{p} \mid \mathcal{L}_{\mathfrak{n}}(\eta,\overline{\eta})$ and $\mathfrak{p} \nmid (\eta-\overline{\eta})^{2}\prod_{i=1}^{n-1}\mathcal{L}_{i}(\eta,\overline{\eta})$ for $\mathfrak{n}>1$. Moreover, a primitive divisor $\mathfrak{q}$ of $\mathcal{L}_{\mathfrak{n}}(\eta,\overline{\eta})$ satisfies $\mathfrak{q} \equiv \left(\frac{(\eta-\overline{\eta})^{2}}{\mathfrak{q}}\right) \pmod{\mathfrak{n}}$, where $\left(\frac{*}{\mathfrak{q}}\right)$ denotes the Legendre symbol \cite{14}.
	
	For $\mathfrak{n}>4$ and $\mathfrak{n} \neq 6$, all $\mathfrak{n}$-th terms of any Lucas sequence $\mathcal{L}_{\mathfrak{n}}(\eta,\overline{\eta})$ possess primitive divisors, except for certain finite values of parameters $\eta, \overline{\eta}$, and $\mathfrak{n}$ \cite{7}.
	\section{Proof of Theorem 1.1}
	For the cases of $\mathfrak{n} = 3$, $\mathfrak{n} = 4$, and $\mathfrak{n} \geq 5$, Equation \eqref{1} will be investigated independently as follows:
	\begin{proposition}
		If $\mathfrak{n} = 3$, the Eq. \eqref{1} has the following solutions:
		\begin{align*}
			(x, y, \alpha, \beta) =& (2, 7, 1, 1), (1232, 115, 3, 1), (23642486, 82375, 9, 1), \\
			&(46, 13, 4, 0), (10, 7, 5, 0)
		\end{align*}
	\end{proposition}
	
	\begin{proof}
		When $\mathfrak{n} = 3$, represent $\alpha = 6\mathfrak{a}_1 + \mathfrak{i}$ and $\beta = 6\mathfrak{b}_1 + \mathfrak{j}$, where $\mathfrak{i}, \mathfrak{j} \in \{0, 1, \ldots, 5\}$. Subsequently, Eq. \eqref{1} takes the form of an elliptic curve 
		\begin{align*}
			L^2 = M^3 - 3^{\mathfrak{i}}113^{\mathfrak{j}}
		\end{align*}
		where $L =\displaystyle \frac{x}{3^{3\mathfrak{a}_1}113^{3\mathfrak{b}_1}}$ and $M = \displaystyle\frac{y}{3^{2\mathfrak{a}_1}113^{2\mathfrak{b}_1}}$. 
		
		Hence, the problem of finding positive integer solutions for Eq. \eqref{1}  is reduced to finding all $\{3,113\}$-integer points on the relevant 36 elliptic curves for every $\mathfrak{i}$ and $\mathfrak{j}$. It's crucial to emphasize that for every finite set of prime numbers $\mathcal{S}$, an $\mathcal{S}$-integer is characterized as a rational number $\displaystyle \frac{r}{s}$ where $r$ and $s > 0$ are relatively prime integers, and any prime factor of $s$ is a member of the set $\mathcal{S}$. 
		
		At this point, we utilized the MAGMA function SIntegralPoints to locate all $\mathcal{S}$-integral points on the given curves. For $\mathcal{S} = \{3, 113\}$ \cite{11}, the identified points are presented below:
		\begin{align*}
			(M, L, \mathfrak{i}, \mathfrak{j}) =& (1, 0, 0, 0), (113, 0, 0, 3), (7, 2, 1, 1), (537, 12444, 2, 1), (3, 0, 3, 0), \\
			&(15, 18, 3, 1), (51, 360, 3, 1), (115, 1232, 3, 1), (303, 5274, 3, 1),  \\
			&\left(\frac{82375}{9}, \frac{23642486}{27}, 3, 1\right), (353103, 209822526, 3, 1), \\
			&(339, 0, 3, 3), (13, 46, 4, 0), (7, 10, 5, 0)
		\end{align*}
		Considering that $x$ and $y$ are positive integers with no common factors, it's important to highlight that only five among these points yield a solution for Eq. \eqref{1}. With this, the proof comes to an end.
	\end{proof} 
	\begin{proposition}
		If $\mathfrak{n} = 4$,  the Eq. \eqref{1} has no positive integer solutions.
	\end{proposition}
	
	\begin{proof}
		Let $\mathfrak{n} = 4$. Initially, express $\alpha = 4{a}_1 + i$ and $\beta = 4{b}_1 + j$, where $i, j \in \{0, 1, 2, 3\}$. Consequently, Eq. \eqref{1} takes the form
		\[
		A^2 = B^4 - 3^{i}113^{j}
		\]
		where $A = \displaystyle \frac{x}{3^{2a_1}113^{2b_1}}$ and $B = \displaystyle \frac{y}{3^{a_1}113^{b_1}}$. 
		
		Identifying all $\mathcal{S} = \{3, 113\}$-integral points on the associated 16 quartic curves corresponds to the search for every integer solution of Eq. \eqref{1}.
		
		By employing the SIntegralLjunggrenPoints, we successfully determined all $\mathcal{S}$-Integral Points on these curves, resulting in
		\[
		(A, B, i, j) = (\mp 1, 0, 0, 0)
		\]
		Given the condition on the values of $x$ and $y$, it is evident that Eq. \eqref{1} does not possess any solutions.
	\end{proof} 
	
	\begin{proposition}
		If $\mathfrak{n} \geq 5$,  the Eq. \eqref{1} does not possess any positive integer solutions.
	\end{proposition}
	
	\begin{proof}
		Suppose that $\mathfrak{n} \geq 5$. If there exists a solution for Eq. \eqref{1} with $\mathfrak{n} = 2^k$ and $k \geq 3$, it can be obtained from solutions with $\mathfrak{n} = 4$ since $y^{2^k} = \left(y^{2^{k-2}}\right)^4$. Consequently, there are no solutions for \eqref{1} with $\mathfrak{n} = 2^k$ and $k \geq 3$. 
		Similarly, \eqref{1} has no solution for $\mathfrak{n} = 3^k$ and $k \geq 2$. Hence, without loss of generality, $\mathfrak{n}$ is an odd prime.\\
		Let's initiate the analysis of the factorization of Eq. \eqref{1} in the field $\mathcal{K} = Q(\sqrt{-d})$ as follows
		\[
		(x + \mathfrak{e} \sqrt{-d})(x - \mathfrak{e} \sqrt{-d}) = y^\mathfrak{n}
		\]
		where $\mathfrak{e} = 3^{\mathfrak{a}}113^{\mathfrak{b}}$ for some integers $\mathfrak{a},\mathfrak{b}\geq 0$ and $d \in \{1, 3, 113, 339\}$. 
		
		Assuming that $y$ is even leads to a contradiction, as $x$ must be odd according to \eqref{1}, resulting in $1 + 3^{\alpha} \equiv 0 \pmod{8}$. 
		
		As a result, $y$ is an odd integer, and therefore, the ideals formed by $x + \mathfrak{e} \sqrt{-d}$ and $x - \mathfrak{e} \sqrt{-d}$ are relatively prime in the field $\mathcal{K}$.
		
		The class number $h(\mathcal{K})$ takes on one of three values: 1, 6, or 8 for the specific choice of $d$.
		
		Thus, we can deduce that the greatest common divisor of $\mathfrak{n}$ and $h(\mathcal{K})$ is 1.
		
		Consider, an algebraic integer $\xi\in \mathcal{K}$ along with units $u_1$ and $u_2$ in the ring of algebraic integers of $\mathcal{K}$, we can express this as:
		\[
		x + \mathfrak{e} \sqrt{-d} = \xi^\mathfrak{n} u_1
		\]
		\[
		x - \mathfrak{e} \sqrt{-d} = \overline{\xi}^\mathfrak{n} u_2
		\]
		Considering the orders of the multiplicative group of units in the ring of algebraic integers of $\mathcal{K}$, which are 2, 4, or 6 based on the value of $d$, and noting that these orders are relatively prime to $\mathfrak{n}$, the presence of units $u_1$ and $u_2$ in the equations can be eliminated. The units $u_1$ and $u_2$ can be integrated into the factors $\xi^\mathfrak{n}$ and $\overline{\xi}^\mathfrak{n}$.
		
		Let us consider the two cases separately, specifically when $d$ belongs to either $\{1, 113\}$ or $\{3, 339\}$ based on distinct integral bases for $\mathcal{O}_{\mathcal{K}}$ as $\{1, \sqrt{-d}\}$ and $\{1, \frac{1 + \sqrt{-d}}{2}\}$, respectively. To begin with, let us assume $d \in \{1, 113\}$. Consequently,
		\[
		x + \mathfrak{e} \sqrt{-d} = \xi^\mathfrak{n} = (s + t \sqrt{-d})^\mathfrak{n}
		\]
		\[
		x - \mathfrak{e} \sqrt{-d} = \overline{\xi}^\mathfrak{n} = (s - t \sqrt{-d})^\mathfrak{n}
		\]
		and $y = s^2 + dt^2$ for some rational integers $s$ and $t$. By analyzing these equations, we can derive that
		\[
		\mathfrak{e} = \mathcal{L}_{\mathfrak{n}} t,
		\]
		where $\mathcal{L}_{\mathfrak{n}} = \displaystyle \frac{\xi^\mathfrak{n} - \overline{\xi}^\mathfrak{n}}{\xi - \overline{\xi}}$. Notably, the sequence $\mathcal{L}_\mathfrak{n}$ is a Lucas sequence. 
		
		The Lucas sequences without primitive divisors are explicitly enumerated in \cite{7}, and it is confirmed that $\mathcal{L}_\mathfrak{n}$ doesn't match any of them. Consequently, we delve into the possibility that a primitive divisor may exist for $\mathcal{L}_\mathfrak{n}$. Assume $\mathfrak{q}$ is any primitive divisor of $\mathcal{L}_\mathfrak{n}$. In this case, $\mathfrak{q}$ is either $3$ or $113$. Considering that any primitive divisor is congruent to $\pm 1$ modulo $\mathfrak{n}$, we rule out the possibility $\mathfrak{q} = 3$ given that $\mathfrak{n} \geq 5$. Thus, we continue with $\mathfrak{q} = 113$. According to the definition of a primitive divisor, $\mathfrak{q} \nmid (\xi - \overline{\xi})^2 = -4dt^2,$ indicating that $d = 1$. Furthermore, since $\left(\displaystyle \frac{-4t^2d}{q}\right) = \left(\displaystyle \frac{-1}{113}\right) = 1$, we deduce that $113 \equiv 1 \pmod{\mathfrak{n}}$ which implies that $\mathfrak{n}=7$. 
		Hence, we derive the expression:
		\begin{eqnarray}
			t(7s^6 - 35s^4t^2 + 21s^2t^4 - t^6) = 3^{\mathfrak{a}}113^{\mathfrak{b}} \label{2}
		\end{eqnarray}
		where $\alpha=2\mathfrak{a}$ and $\beta=2\mathfrak{b}$. Given that $113 \nmid t$ and $\gcd(s, t) = 1$, we conclude that either $t = \pm 1$ or $t = \pm 3^{\mathfrak{a}}$.
		If $t = \pm 1$, then Eq. \eqref{2} transforms into the following equations:
		\[
		7s^6 - 35s^4 + 21s^2 - 1= \pm 3^{\mathfrak{a}}113^{\mathfrak{b}}
		\]
		By letting $\mathfrak{a} = 2a_1 + i$ and $\mathfrak{b} = 2b_1 + j$ for $i, j \in \{0, 1\}$, we obtain:
		\[
		7s^6 - 35s^4 + 21s^2 - 1= \delta Y^{2},~\text{where}~ \delta = \pm 3^{i}113^{j}~\text{and}~Y = 3^{ a_1}113^{b_1}.
		\]
		By multiplying $7^{2}\delta^{3}$ both sides of the above equation, we obtain:
		\[
		V^3 - 35\delta V^2 + 147\delta^{2} V^{2} - 49\delta^{3} = (7\delta^{2}Y)^2, ~\text{where}~V = 7\delta s^2.
		\] 
		We utilize MAGMA to compute all integral points of these elliptic curves for each value of $\delta$, and find the points:
		\begin{align*}
			(V, 7\delta^{2}Y) &= (1, 8), (58, 293) \ \text{for } \delta = 1, \\
			(V, 7\delta^2Y) &= (226, 12769) \ \text{for } \delta = 113, \\
			(V, 7\delta^2Y) &= (7, 56), (91, 56), (3892, 239519) \ \text{for } \delta = 3, \\
			(V, 7\delta^2Y) &= (13195, 672728), (13447, 715064) \ \text{for } \delta = 339, \\
			(V, 7\delta^2Y) &= (-21, 56), (-5, 8), (0, 7), (7, 56), (39, 344) \ \text{for } \delta = -1, \\
			(V, 7\delta^2Y) &= (-41, 568), (1243, 102152) \ \text{for } \delta = -113, \\
			(V, 7\delta^2Y) &= (-2147, 102152), (2230, 331171) \ \text{for } \delta = -339.
		\end{align*}
		It can be verified that none of them leads to a solution of Eq. \eqref{1}.
		
		In the case of $t = \pm 3^{\mathfrak{a}}$, dividing both sides of Eq. \eqref{2} by $t^6$ and introducing $\mathfrak{b} = 2b_1 + j$ for $j \in \{0, 1\}$ yields:
		\[
		7\left(\frac{s^2}{t^2}\right)^{3} - 35\left(\frac{s^2}{t^2}\right)^{2} + 21\left(\frac{s^2}{t^2}\right) - 1 = \pm 113^j \left(\frac{113^{b_1}}{t^3}\right)^{2}.
		\]
		Setting $\delta = \pm 113^j$, $\displaystyle Y = \frac{113^{b_1}}{t^3}$, and multiplying both sides by $7^2\delta^3$, we arrive at:
		\[
		V^3 - 35\delta V^2 + 147\delta^2 V - 49\delta^3 = (7\delta^{2}Y)^2, ~\text{where}~ V = \frac{7\delta s^2}{t^2}.
		\]
		Utilizing MAGMA, we determined all $\{3\}$-Integral Points on the relevant elliptic curves, resulting in the following set of points:
		\begin{align*}
			~~~~~~&\text{For } \delta = 1, \ (V, 7\delta^{2}Y) = (1, 8), (58, 293) \\
			~~~~~~&\text{For } \delta = 113, \ (V, 7\delta^2Y) = (226, 12769) \\
			~~~~~~&\text{For } \delta = -1, \ (V, 7\delta^2Y) = (-21, 56), \left(\frac{-77}{9}, \frac{728}{27}\right), (-5, 8), (0, 7), (7, 56), (39, 344) \\
			~~~~~~&\text{For } \delta = -113, \ (V, 7\delta^2Y) = (-41, 568), (1243, 102152)
		\end{align*}
		Equation \eqref{1} does not find a solution through any of these cases.
		
		Now, consider the assumption that $d$ belongs to the set $\{3, 339\}$. In such instances, the integral basis for $\mathcal{O}_\mathcal{K}$ is $\{1, \frac{1 + \sqrt{-d}}{2}\}$, and as a result, we express
		\[
		x + \mathfrak{e}\sqrt{-d} = \xi^\mathfrak{n} = \left(\frac{s + t\sqrt{-d}}{2}\right)^\mathfrak{n}
		\]
		\[
		x - \mathfrak{e}\sqrt{-d} = \overline{\xi}^\mathfrak{n} = \left(\frac{s - t\sqrt{-d}}{2}\right)^\mathfrak{n}
		\]
		with $s \equiv t \pmod{2}$. Then
		\[
		2\mathfrak{e} = \mathcal{L}_\mathfrak{n} t,
		\]
		where $\mathcal{L}_{\mathfrak{n}} = \displaystyle \frac{\xi^\mathfrak{n} - \overline{\xi}^\mathfrak{n}}{\xi - \overline{\xi}}$. 
		
		Observing that $\mathcal{L}_\mathfrak{n}$ is a Lucas sequence, if there is no primitive divisor for $\mathcal{L}_\mathfrak{n}$, it implies that $\mathcal{L}_\mathfrak{n}$ corresponds to one of the Lucas sequences outlined in the table in \cite{7}. However, this observation does not hold true for $d \in \{3, 339\}$. 
		
		Hence, $\mathcal{L}_\mathfrak{n}$ possesses a primitive divisor, denoted as $\mathfrak{q}$, where $\mathfrak{q}$ can either be $3$ or $113$.
		
		From the properties $\mathfrak{q} \nmid (\xi - \overline{\xi})^2 = -dt^2$ and $\mathfrak{q} \equiv \left(\displaystyle \frac{-t^2d}{\mathfrak{q}}\right) \pmod{\mathfrak{n}}$, it follows that $\mathfrak{q} = 113$, $d = 3$, and $\mathfrak{n} = 19$. This contradicts the assumption that $19 \nmid \mathfrak{n}$. This completes the proof.
	\end{proof}
	
\end{document}